\documentclass[12pt,reqno]{amsart}
\usepackage[margin=1.25in]{geometry}
\usepackage{amssymb}
\usepackage{amsfonts} 
\usepackage{latexsym}   
\usepackage{amssymb}    
\usepackage{amsmath}    
\usepackage{amsbsy}
\usepackage{amsgen}
\usepackage{amsfonts}
\usepackage{array}

\usepackage{epsfig}
\usepackage{psfrag}
\usepackage[all]{xy}
\usepackage{hyperref}

\usepackage{xcolor}
\usepackage{amssymb}
\usepackage{mathabx}
\usepackage{amsfonts}
\usepackage{hyperref}
\usepackage{amsmath}
\usepackage{graphicx}
\usepackage[all]{xy}
\providecommand{\U}[1]{\protect\rule{.1in}{.1in}}

\theoremstyle{plain}

\newtheorem{algorithm}{Algorithm}[section]

\newtheorem{definition}[algorithm]{Definition}

\newtheorem{lemma}[algorithm]{Lemma}

\newtheorem{theorem} [algorithm] {Theorem}

\newtheorem{theoremlet'}[thm]{Theorem$'$}

\newtheorem*{H-co}{h-Cobordism Theorem}

\usepackage{amssymb}
\usepackage{amsfonts} 
\usepackage{latexsym}   
\usepackage{amssymb}    
\usepackage{amsmath}    
\usepackage{amsbsy}
\usepackage{amsgen}
\usepackage{amsfonts}
\usepackage{array}
\usepackage{scalerel}

\usepackage{epsfig}
\usepackage{psfrag}
\usepackage[all]{xy}
\usepackage{amssymb}
\usepackage{mathabx}

\newtheorem{observe}[algorithm]{Observation}
\newtheorem{remark}[algorithm]{Remark}
\newtheorem{proposition}[algorithm]{Proposition}

\newtheorem*{observe*}{Observation}

\DeclareMathOperator{\Fix}{Fix}
\def\codim{\textrm{codim}}
\def\bdm{\begin{displaymath}}
\def\edm{\end{displaymath}}
\def\beq{\begin{equation}}
\def\eeq{\end{equation}}
\def\bes{\begin{equation*}}
\def\ees{\end{equation*}}
\def\epcm{\end{picture}\end{center}\end{minipage}}
\def\bpcm{\begin{minipage}{80pt}\begin{center}\begin{picture}}

\def\Z{\mathbb{Z}}
\def\RP{\mathbb{R}\mathrm{P}}
\def\t2{T^2}

\def\f4{F_4}
\def\g2{G_2}

\def\p2{\frac{\pi}{2}}

\def\Fix{\textrm{Fix}}

\def\dim{\textrm{dim}}

\def\CP{\mathbb{C}\mathrm{P}}

\def\RP{\mathbb{R}\mathrm{P}}

 \numberwithin{equation}{section}
  \numberwithin{figure}{section}

\usepackage{amsthm}
\usepackage{amsmath}
\usepackage[margin=1.25in]{geometry}
\usepackage{amssymb}
\usepackage{verbatim}
\usepackage{longtable,enumitem,tikz-cd,color}
\usepackage{stmaryrd}
\usepackage{graphicx}
\usepackage{latexsym}   
\usepackage{amssymb}    
\usepackage{amsbsy}
\usepackage{array}
\usepackage[all]{xy}
\usepackage{hyperref}
\usepackage{caption}
\usepackage{mathabx}
\usepackage{amsgen}
\usepackage{scalerel}
\usepackage{epsfig}
\usepackage{psfrag}
\usepackage{wrapfig}
\newtheorem*{acknowledge}{Acknowledgements}

\newtheorem*{Borel}{Borel Formula}

\newtheorem*{CL}{Connectedness Lemma}

\newtheorem*{A}{Theorem A}
\newtheorem*{far}{Corollary B}

\newcommand{\ep}{\epsilon}

\DeclareMathOperator{\diag}{diag}

\newcommand{\embedded}{\hookrightarrow}

\usepackage[T1]{fontenc}
\usepackage{makecell}

\begin{document}

\title[$\mathbb{Z}_{2}$-torus actions on positively curved manifolds]{$\mathbb{Z}_{2}$-torus actions on positively curved manifolds}
\author[Ghazawneh]{Farida Ghazawneh}
\address[Ghazawneh]{Department of Mathematics, Statistics, and Physics, Wichita State University, Wichita, Kansas}
\email{ftghazawneh@shockers.wichita.edu}


\date{\today}

\begin{abstract}

Kennard, Khalili Samani, and Searle showed that for a 
$\Z_2$-torus acting on a closed, positively curved Riemannian $n$-manifold, $M^{n}$, with a non-empty fixed point set for $n$ large enough and $r$ approximately half the dimension of $M$, then $M^n$ is homotopy equivalent to $S^n$, $\RP^n$, $\CP^{\frac{n}{2}}$, or a lens space. In this paper, we lower $r$ to approximately $2n/5$ and show that we still obtain the same result.
\end{abstract}

\maketitle

\section{Introduction}
 
 To date, other than some special examples in  dimensions less than or equal to $24$, all known examples of closed positively curved manifolds are spherical in nature and are highly symmetric.
The Symmetry Program suggests that one may approach the classification of such manifolds by considering those with "large" symmetries.  Indeed,  over the last 30 years this has been a fruitful means of tackling the classification problem.  Since the group of isometries of a closed manifold is a compact Lie group and every continuous, compact Lie group contains a maximal torus, the case of torus actions have naturally been studied extensively, see, for example, work of Grove and Searle~\cite{GS}, Rong~\cite{R}, Fang and Rong~\cite{FR05}, Wilking~\cite{wilking2003torus}, as well as more recent work of Kennard, Weimeler, and Wilking~\cite{kennard2021splitting, kennard2022positive}. In turn, since a torus is a product of circles and therefore contains a product of cyclic groups, it is natural to also consider $\Z_k^r$-actions. Previous work on positively and non-negatively curved manifolds with discrete symmetries of dimension $4$ can be found in Yang \cite{Yan94}, Hicks \cite{Hic97}, Fang \cite{Fan08}, and Kim and Lee \cite{KL09}, and of higher dimensions in Fang and Rong \cite{FR04}, Su and Wang \cite{SW08},  Wang \cite{W10}, and more recently Kennard, Khalili Samani and Searle \cite{KKSS}. For all but the last of these results, a cyclic group  or a product of cyclic groups has been chosen so that the order of each cyclic group is large enough to produce a fixed point.

In~\cite{KKSS}, they consider $\Z_2^r$-actions with the additional assumption that there exists a fixed point.
 We also make the assumption that a $\Z_2^r$-action has a fixed point and show that the lower bound of approximately $n/2$ obtained in Theorem B of \cite{KKSS} (Theorem~\ref{5.3} in this paper) can be improved to approximately $2n/5$ in the following theorem.

\begin{A}\label{A}

Assume $\mathbb{Z}_{2}^{r}$ acts effectively and isometrically on a closed, positively curved manifold $M^{n}$ with a non-empty fixed-point set with $n\geq15$. If  $$r\geq.3992n+1.5\log_{2}n+2.772,$$ then $M$ is homotopy equivalent to $S^n$, $\RP^n$, $\CP^{n/2}$, or a lens space. 
 
\end{A}

\noindent Since $2n/5>.3992n+1.5\log_{2}n+2.772$ for $n\geq31,482$, we obtain the following corollary.

\begin{far}\label{corollaryB}
    Assume $\mathbb{Z}_{2}^{r}$ acts effectively and isometrically on a closed, positively curved manifold $M^{n}$ with a non-empty fixed-point set. If $n\geq31,482$ and $r\geq\frac{2n}{5}$, then $M$ is homotopy equivalent to $S^n$, $\RP^n$, $\CP^{n/2}$, or a lens space. 
\end{far}

\begin{remark}
  Theorem A of~\cite{KKSS} states that an isometric $\Z_2^r$-action on a closed, positively curved manifold, with a fixed point is not effective if $r>n$. In particular, for $n<15$, we have $.3992n+1.5\log_{2}n+2.772>n$ and so we do not consider these dimensions. For $n\leq126$, the bound given in Theorem~\ref{5.3} is better than the one given in Theorem A.
  \end{remark}

\subsection{Organization}
The paper is organized as follows.
In Section \ref{s2}, we gather background material needed to prove \hyperref[A]{Theorem A} and in Section \ref{s3}, we prove \hyperref[A]{Theorem A}.

\begin{acknowledge}

The author is extremely grateful to her PhD advisor, Catherine
Searle, for suggesting this area of study to her and for her
mathematical guidance and support. The author also thanks Lee Kennard for helpful conversations and for the suggestion to consider this problem. Finally, she gratefully acknowledges partial support by Catherine Searle's NSF Grants DMS-1906404 and DMS-2204324. This work forms a part of the author's thesis.

\end{acknowledge}

\section{Preliminaries}\label{s2}
In this section we gather background material used in the proof of the main theorem.

\subsection{Error Correcting Codes}

We use the theory of error correcting codes to obtain bounds on the codimension of fixed-point sets of involutions. Before we begin, we recall the definition of the Hamming weight and the Hamming distance.

\begin{definition} [\bf{Hamming weight}] Let $\iota\in\mathbb{Z}_{2}^{n}$ be an involution. The Hamming weight of $\iota$, denoted by $|\iota|$, is the number of non-trivial entries in $\iota$. The Hamming distance between two involutions $\iota_1$, $\iota_{2}\in\Z_2^n$ is $|\iota_{1}\iota_{2}|$.
    
\end{definition}
An error correcting code can be thought of as an injective linear map $\Z_2^r\rightarrow\Z_2^n$. For an involution $\iota$ acting isometrically on a manifold $M$, let $N$ be a fixed-point set component of $\iota$. The codimension of $N$ is equal to the Hamming weight of $\iota$, where $\iota$ is considered as a linear map on $T_{x}M$ for some $x\in N$. The following theorem from~\cite{KKSS}, which follows from Proposition 3.1 in Wilking~\cite{wilking2003torus}, is an improvement of the Elias-Bassalygo bound for error correcting codes.

\begin{theorem}\label{1.2}\cite{wilking2003torus} Fix $\delta\in(0,\frac{1}{2})$. If $\rho:\mathbb{Z}_{2}^{r}\rightarrow\mathbb{Z}_{2}^{n}$ is an injection, that is not necessarily linear, such that $|\rho(\iota)|\geq\delta n$ for all non-trivial $\iota\in\mathbb{Z}_{2}^{r}$, then $$r\leq(1-H(J(\delta)))n+1.5\log_{2}n+0.5\log_{2}(\frac{1-J(\delta)}{J(\delta)})+1.5,$$ where $J(\delta)=\frac{1}{2}(1-\sqrt{1-2\delta})$ and $H(\epsilon)=-\epsilon\log_{2}\epsilon-(1-\epsilon)\log_{2}(1-\epsilon).$
    
\end{theorem}
\begin{remark}
 The function $J(\delta)$ is an increasing function that maps the interval $[0,\frac{1}{2}]$ onto itself, and is related to the Johnson bound. The entropy function $1-H(\epsilon)$ is a decreasing function on $[0,\frac{1}{2}]$.
\end{remark}

\subsection{Topological and Geometric preliminaries} In this subsection, we recall some useful topological and geometric results. We begin by recalling the Connectedness Lemma from~\cite{wilking2003torus}.

\begin{CL}\label{CL} \cite{wilking2003torus}
Let $M^n$ be a closed Riemannian manifold with positive sectional curvature. The following hold:
       \begin{enumerate}[font=\normalfont]
       \item If $N^{n-k}$ is a closed, totally geodesic submanifold of $M$, then the inclusion map $N^{n-k}\embedded M^n$ 
       is $(n-2k+1)$-connected.
       \item If $N_1^{n-k_1}$ and $N_2^{n-k_2}$ are closed, totally geodesic submanifolds of $M$ with $k_1 \leq k_2$ and $k_1 +k_2\leq n$, then 
       the inclusion $N_1^{n-k_1}\cap N_2^{n-k_2}\embedded N_2^{n-k_2}$ is $(n-k_1-k_2)$-connected.
       \end{enumerate}
\end{CL}

The next result, a reformulation of the proof of Lemma 6.2 from~\cite{wilking2003torus}, is an application of the~\hyperref[CL]{Connectedness Lemma}. It states that if the codimension of a submanifold, $N\subset M$, is bounded above by approximately one fourth the dimension of $M$ and the cohomology ring of $N$ is of a certain type, then we obtain information about the cohomology ring of $M$.

\begin{theorem}\label{wilking}\cite{wilking2003torus}
Let $N^{n-k} \subseteq M^n$ be a totally geodesic embedding of a closed, simply connected, positively curved manifold such that $k \leq \tfrac{n+3}{4}$. If $N$ has the cohomology ring of a sphere, complex projective space, or quaternionic projective space, then the same holds for $M$. Moreover, $k$ is even in the case of complex projective spaces, and $k$ is divisible by four in the case of quaternionic projective spaces.
\end{theorem}

We now recall some useful results from~\cite{KKSS}. The first of these, Proposition 3.3 in~\cite{KKSS}, is a variation on the situation in Theorem~\ref{wilking}. Namely one assumes the existence of two closed, totally geodesic, submanifolds $N_{1}$ and $N_{2}$, one of which has periodic cohomology. By making additional assumptions on the codimension of $N_{1}$ and $N_{2}$, we conclude that $M^n$ also has periodic cohomology.

\begin{proposition}\label{3.3}\cite{KKSS}
Let $M^n$ be a closed, simply connected, positively curved Riemannian manifold. Suppose that $N_1^{n-k_1}$ and $N_2^{n-k_2}$ are two closed, totally geodesic submanifolds of $M$ with $k_1\leq k_2$. Let $R$ be $\Z$ or $\Z_p$, for some prime $p$. The following hold:
       \begin{enumerate}
       \item Suppose that $4k_1\leq n+3$ and $k_1+2k_2\leq n+1$. If $N_2$ is an $R$-cohomology sphere, then so is $M$.
       \item Suppose that $4k_1\leq n+3$ and $g + k_1 + 2k_2 \leq n + 3$ for some $g \geq 2$ that divides $k_1$. If $N_2$ has $g$-periodic $R$-cohomology, then so does $M$.
       \end{enumerate}
\end{proposition}

Since Theorem~\ref{wilking} and Proposition~\ref{3.3} involve simply-connected manifolds, we now recall Theorem 3.4 from~\cite{KKSS}, which shows that information on the fundamental group of the manifold and the cohomology ring of its universal cover can lead to information about homotopy equivalence.

\begin{theorem}\label{CtoHE}\cite{KKSS}
Let $M^n$ be a closed, smooth manifold, and let $\widetilde M^n$ denote the universal cover. The following hold:
	\begin{enumerate}
	\item If $\pi_1(M)$ is cyclic and $\widetilde M$ is a cohomology sphere, then $M$ is homotopy equivalent to $S^n$, $\RP^n$, or $S^n/\Z_l$ for some $l \geq 3$. In the last case, $n$ is odd; and 
	\item If $M$ is simply connected and has the cohomology of $\CP^{\frac n 2}$, then $M$ is homotopy equivalent to $\CP^{\frac n 2}$.
	\end{enumerate}
\end{theorem}

 Another useful tool when working with fixed point sets of discrete actions, is the Borel Formula.

\begin{Borel}\label{borel}\cite{Bor}
    Let $\Z_p^r$ act smoothly on $M$, a Poincar\'e duality space, with fixed-point set component $F$. Then
	$$\textup{codim}(F \subseteq M) = \sum \textup{codim}(F \subseteq F')$$
where the sum runs over fixed-point set components $F'$ of corank one subgroups $\Z_p^{r-1} \subseteq \Z_p^r$ for which $\textup{dim}(F') > \textup{dim}(F)$.

These subgroups are the kernels of the irreducible subrepresentations. In particular, the number of pairwise inequivalent irreducible subrepresentations of the isotropy representation at a normal space to $F$ is at least $r$ if the action is effective. Moreover, equality holds only if the isotropy representation is equivalent to one of the form
	$$(\ep_1,\ldots,\ep_r) \mapsto \diag(\ep_1 I_{m_1}, \ldots, \ep_r I_{m_r})$$
where $\ep_i \in \{\pm 1\}$, $m_i > 0$ denote the multiplicities, $I_m$ denotes the identity matrix of rank $m$.
\end{Borel}

 The following useful observation from~\cite{KKSS}, is a direct consequence of the Borel formula.
\begin{observe}\label{observation} Let $F_j^{m_j}$ be the fixed-point set of a 
$\Z_2^{r-j}$-action by isometries on a closed, positively curved manifold, and suppose that $r-j\geq 2$. If the isotropy
representation has exactly $r-j$ irreducible subrepresentations, and if the fixed-point set component $F_{j+1}$ containing $F_j$ of some $\Z_2^{r-j-1}$ has the property that the codimension $k_j$ of the inclusion $F_j \subseteq F_{j+1}$ is minimal, then this inclusion is $\textup{dim}(F_j)$-connected.
\end{observe}

In the next lemma, we summarize the relevant results from the Codimension One and Codimension Two Lemmas in~\cite{KKSS}, used in the proof of Theorem~\ref{5.3}.

\begin{lemma}\label{1PLUS2}\cite{KKSS} Let $M^{n}$ be a closed, positively curved Riemannian manifold. Suppose $N^{n-k}$ is a closed, totally geodesic submanifold of $M$ of codimension one or two. Then the following hold:
\begin{enumerate}
    \item If $k=1$, then $N$ is homotopy equivalent to $S^{n-1}$ or $\RP^{n-1}$; or
    \item If $k=2$ and $N$ is $(n-2)$-connected, then $N$ is homotopy equivalent to $S^{n-2}$, $\RP^{n-2}$, $\CP^{\frac{n-2}{2}}$, or a lens space.
\end{enumerate}
In particular, if $N$ is homotopy equivalent to one of $S^{n-k}$, $\RP^{n-k}$, $\CP^{\frac{n-k}{2}}$, or a lens space, then $M$ is homotopy equivalent to one of $S^{n}$, $\RP^{n}$, $\CP^{\frac{n}{2}}$, or a lens space, respectively.
    
\end{lemma}

 Finally, we recall one of the main results of~\cite{KKSS}, Theorem 5.3 (see also Theorem B in~\cite{KKSS}). It classifies up to homotopy equivalence, closed, positively curved manifolds with a $\Z_2$-torus action of rank approximately half the dimension of $M$. Before we state it, we set the following notation: let $B=\{4,12\}\cup\{3,7,11,23\}$ and define $\delta_B(n)$ to be one for $n\in B$ and zero otherwise.
\begin{theorem}\label{5.3}\cite{KKSS}
Assume $\mathbb{Z}_{2}^{r}$ acts effectively and isometrically on a closed, positively curved manifold $M^{n}$ with a non-empty fixed-point set with $n\geq2$. If $$r>\frac{n}{2}+\delta_B(n),$$ then one of the following holds:
\begin{enumerate}
    \item $M$ is homotopy equivalent to $S^n$ or $\RP^n$.
    \item $M$ is homotopy equivalent to $\CP^{\frac{n}{2}}$, $r=\frac{n}{2}+1$ and $n\notin B$.
    \item $M$ is homotopy equivalent to $S^{n}/\mathbb{Z}_l$ for some $l\geq3$, $r=\frac{n+1}{2}$, and $n\notin B$.
\end{enumerate}
    
\end{theorem}

\section{The proof of Theorem A}\label{s3} We begin by establishing an important tool for the proof of~ \hyperref[A]{Theorem A}, which can be used to find two distinct involutions fixing closed submanifolds of sufficiently small codimension.

\begin{lemma}\label{mainlemma} Let $M^n$, $n\geq15$, be an $n$-dimensional Riemannian manifold admitting an effective, isometric $\mathbb{Z}_{2}^{r}$-action with a fixed point $x$. If $r\geq.3992n+1.5\log_{2}n+2.772$, then
\begin{enumerate}
\item there exists an involution $\iota_{1}\in\mathbb{Z}_{2}^{r}$ such that $\textup{codim}(M_{x}^{\iota_{1}})\leq\frac{n+3}{4}$; and
 \item there exists an involution $\iota_{2}\in\mathbb{Z}_{2}^{r}$ such that $\textup{codim}(M_{x}^{\iota_{2}})\leq\frac{n+1}{3}$, where
 
 $\iota_{2}\notin\langle\iota_{1}\rangle.$
\end{enumerate}

  \end{lemma}
  
\begin{proof}
Consider the isotropy representation $\rho:\mathbb{Z}_{2}^{r}\rightarrow\mathbb{Z}_{2}^{n}$ at $x$. Recall that the Hamming weight of the image of an involution equals the codimension of its fixed-point set.

To prove Part (1), we assume that $\codim(M_{x}^{\iota_{1}})>\frac{n+3}{4}$ for all non-trivial $\iota_{1}\in\mathbb{Z}_{2}^{r}$, to derive a contradiction. By assumption, $|\rho(\iota_{1})|\geq\frac{n}{4}$ for all non-trivial $\iota_{1}\in\mathbb{Z}_{2}^{r}$, so we may apply Theorem~\ref{1.2} with $\delta=\frac{1}{4}$. This, together with a calculation, gives that
\begin{alignat*}{2}
r &\le(1-H(J(\frac{1}{4})))n+1.5\log_{2}n+0.5\log_{2}(\frac{1-J(\frac{1}{4})}{J(\frac{1}{4})})+1.5  \\
&\le0.3991n+1.5\log_{2}n+2.7715.
\end{alignat*}

\noindent Combining this inequality with our original hypothesis on $r$, we conclude that $n<0$, giving us a contradiction. This completes the proof of Part (1).

To prove Part (2), we assume that $\codim(M_{x}^{\iota_{2}})>\frac{n+1}{3}$ for all non-trivial $\iota_{2}\in\mathbb{Z}_{2}^{r}\,\setminus\,\langle\iota_{1}\rangle$ to derive a contradiction. By assumption, $|\rho(\iota_{2})|\geq\frac{n}{3}$ for all non-trivial $\iota_{2}\in\mathbb{Z}_{2}^{r}\,\setminus\,\langle\iota_{1}\rangle$, so we may apply Theorem~\ref{1.2} with $\delta=\frac{1}{3}$. A computation gives us that
$$r\leq0.256n+1.5\log_{2}n+2.4495.$$

\noindent Combining this inequality with our original assumption on $r$, once again we see that $n<0$, a contradiction. This completes the proof of Part (2) and of the lemma.
\end{proof}

\begin{remark}\label{rmk2}In the situation of Lemma~\ref{mainlemma}, if $\pi_{1}(M)=0$, $N_{1}^{n-k_{1}}\subset M_{x}^{\iota_{1}}$, and $N_{2}^{n-k_{2}}\subset M_{x}^{\iota_{2}}$, then $k_{1}$ and $k_{2}$ satisfy the hypothesis of Parts (1) and (2) of Proposition~\ref{3.3}.
    
\end{remark}

We are now ready to prove~\hyperref[A]{Theorem A}, which we restate for the convenience of the reader.

\begin{theorem}Assume $\mathbb{Z}_{2}^{r}$ acts effectively and isometrically on a closed, positively curved manifold $M^{n}$ with a non-empty fixed-point set with $n\geq15$. If  $$r\geq.3992n+1.5\log_{2}n+2.772,$$ then $M$ is homotopy equivalent to $S^n$, $\RP^n$, $\CP^{n/2}$, or a lens space.  

\end{theorem}

\begin{proof} The proof is by total induction on the dimension. The anchor of the induction is for dimensions $\leq126$.

For $n\leq126$, we can see that $.3992n+1.5\log_{2}n+2.772>\frac{n}{2}+\delta_B(n)$, and thus our anchor holds by Theorem~\ref{5.3}. So, we assume $n>126$. We suppose the result holds for all dimensions less than $n$ and show it also holds for dimension $n$. Let $x$ be a fixed point of the $\Z_2^r$-action. 

By Part (1) of Lemma~\ref{mainlemma}, there exists a non-trivial involution $\iota_{1}\in\mathbb{Z}_{2}^{r}$ and $x\in N_{1}^{n-k_{1}}\subset\Fix(M;\iota_{1})$ such that $k_{1}\leq\frac{n+3}{4}$ and $N_{1}$ is of maximal dimension. By maximality of $N_{1}$, the kernel of the $\Z_2^r$-action is $\Z_2$ and hence $N_{1}$ admits an effective $\Z_2^{r-1}$-action with $x\in N_{1}$.

By Part (2) of Lemma~\ref{mainlemma}, there exists a second non-trivial involution $\iota_{2}\in\mathbb{Z}_{2}^{r}$ such that $\iota_2\notin\langle\iota_{1}\rangle$ and $x\in N_{2}^{n-k_{2}}\subset\Fix(M;\iota_{2})$ such that $k_{2}\leq\frac{n+1}{3}$ and $N_{2}$ is of maximal dimension. By maximality of $N_{2}$, the kernel of the $\Z_2^r$-action is $\Z_2$ and hence $N_{2}$ admits an effective $\Z_2^{r-1}$-action with $x\in N_{2}$. Without loss of generality, we assume $k_{1}\leq k_{2}$. We now proceed to argue according to the codimension of $N_{2}$. We have two cases: Case 1, where $k_{2}$ is greater than or equal to 3 and Case 2, where $k_{2}$ is less than or equal to 2. We begin with Case 1.

\noindent \textbf{Case 1: \boldmath{$k_2\geq3$}.} Recall that $\mathbb{Z}_{2}^{r-1}$ acts effectively and isometrically on $N_{2}^{n-k_{2}}$.

Using the fact that $r\geq.3992n+1.5\log_{2}n+2.772$ and $k_{2}\geq3$, a computation gives
$$r-1\geq.3992(n-k_{2})+1.5\log_{2}(n-k_{2})+2.772.$$  

\noindent Hence, the induction hypothesis is satisfied and it then follows that $N_2$ is homotopy equivalent to $S^{n-k_{2}}$, $\RP^{n-k_{2}}$, $\CP^{\frac{n-k_{2}}{2}}$, or a lens space. We lift to the universal cover of $M$ via the covering map $\pi:\widetilde{M}\rightarrow M$. Let $\bar{N}_{2}$ be the lift of $N_{2}$. Since $k_{2}\leq\frac{n+1}{3}$, it is easy to see by the~\hyperref[CL]{Connectedness Lemma}, that both $N_{2}$ and $\bar{N}_{2}$ are at least 5-connected for $n\geq126$.  So $\bar{N}_{2}$ is simply-connected and therefore is the universal cover of $N_{2}$. By Remark~\ref{rmk2}, we see that the codimension of $N_{1}$ and $N_{2}$ satisfy the hypothesis of Proposition~\ref{3.3}.

In the case where $\bar{N}_2$ is a cohomology sphere, then by Part (1) of Proposition~\ref{3.3} so is $\widetilde{M}$. Recall that $N_{2}$ is at least 5-connected and since $\pi_{1}(N_2)$ is cyclic in this case, it follows that $\pi_{1}(M)$ is cyclic. By Part (1) of Theorem~\ref{CtoHE}, $M$ is homotopy equivalent to $S^n$, $\RP^n$, or a lens space. 

If $\bar{N}_2$ is a cohomology complex projective space, then by Part (2) of Proposition~\ref{3.3} with $g=2$, so is $\widetilde{M}$. Since $\pi_{1}(N_2)=0$ in this case, it follows as in the previous case that $\pi_{1}(M)=0$. By Part (2) of Theorem~\ref{CtoHE}, we then conclude that $M$ is homotopy equivalent to $\CP^{\frac{n}{2}}$.

\noindent\textbf{Case 2: \boldmath{$k_2\leq2$}.} Since $k_1\leq k_2\leq2$, we have two subcases to consider: Case 2.a, where one of $k_1$ or $k_2$ is equal to one, Case 2.b, where there are two irreducible representations and $k_1=k_2=2$. We break Case 2.b into two further subcases: Case 2.b.i, where the submanifolds $N_1$ and $N_2$ intersect transversely and Case 2.b.ii, where the submanifolds $N_1$ and $N_2$ intersect non-transversely.

\noindent\textbf{Case 2.a: \boldmath{$k_{i}=1$} for some $i\in\{1,2\}$.} Part (1) of Lemma~\ref{1PLUS2} gives us that $M$ is homotopy equivalent to $S^{n}$ or $\RP^{n}$.

\noindent\textbf{Case 2.b.i: \boldmath{$k_1=k_2=2$} and $N_1$ and $N_2$ intersect transversely.} By Observation~\ref{observation}, $N_{1}\cap N_{2}$ is $\dim(N_{1}\cap N_{2})$-connected in $N_2$. It then follows from Part (2) of Lemma~\ref{1PLUS2} that $M$ is homotopy equivalent to $S^{n}$, $\RP^{n}$, $\CP^{\frac{n}{2}}$ or a lens space.

\noindent\textbf{Case 2.b.ii\boldmath{: $k_1=k_2=2$} and the intersection is not transverse.} Since $N_{1}$ and $N_{2}$ intersect non-transversely, then the codimension of $N_{1}\cap N_{2}$ is strictly less than $k_{1}+k_{2}=4$. In particular, $\codim(N_{1}\cap N_{2})=3$. It follows from Part (1) of Lemma~\ref{1PLUS2} that $N_{2}$ is homotopy equivalent to $S^{n-k_{2}}$ or $\RP^{n-k_{2}}$. In the case where $N_{2}$ is homotopy equivalent to $\RP^{n-k_{2}}$, we lift to the universal cover first, as we did in Case 1 and then use Part (1) of Theorem~\ref{CtoHE} to conclude that $M$ is homotopy equivalent to $\RP^{n}$ or a lens space. In the case where $N_{2}^{n-k_{2}}$ is homotopy equivalent to $S^{n-k_{2}}$, since $k_{2}=2\leq\frac{n+3}{4}$, we use Theorem~\ref{wilking} to conclude that $M$ is homotopy equivalent to $S^{n}$. This concludes the proof of~\hyperref[A]{Theorem A}.
\end{proof}

\end{document}